\documentclass[11pt]{article}

\usepackage{amsmath,amsfonts,amsthm,amssymb,mathtools}
\usepackage{mathrsfs}
\usepackage{graphicx}
\usepackage{xcolor}
\usepackage[numbers]{natbib}
\usepackage{enumitem}
\usepackage[colorlinks=true,
linkcolor=blue,citecolor=blue,urlcolor=blue]{hyperref}

\allowdisplaybreaks

\newtheorem{theorem}{Theorem}
\newtheorem{lemma}{Lemma}
\newtheorem{proposition}{Proposition}
\newtheorem{corollary}{Corollary}
\newtheorem{remark}{Remark}
\newtheorem{example}{Example}
\newtheorem{claim}{Claim}

\newtheorem{fact}{Fact}
\title{\bfseries\Large
	Equality Cases for the Face-Degree Majorization Theorem
	on Simplicial Complexes}

\author{
	Yueli Han,\ \ Lu Lu\footnote{
		Corresponding author.
		E-mail addresses:
		hanyueli2022@163.com (Y. Han),
		lulumath@csu.edu.cn (L. Lu).}\\[2mm]
	\footnotesize School of Mathematics and Statistics, HNP-LAMA,
	Central South University\\
	\footnotesize Changsha, Hunan 410083, China
}

\date{}

\begin{document}
	
	\maketitle

\begin{abstract}
	The Grone--Merris--Bai theorem states that the Laplacian spectrum of a
	simple graph is majorized by its conjugate degree sequence. Recently,
	Zhang, Song, and Fan extended this result to simplicial complexes by
	establishing a majorization relation between the spectrum of the
	$(r-1)$-dimensional up-Laplacian and the conjugate $(r-1)$-degree sequence.
	In this paper, we characterize all equality cases in the partial-sum
	inequalities of this higher-dimensional majorization theorem. For every
	$r$-dimensional simplicial complex $X$ with $r\ge2$, we prove that
	\[
	\sum_{i=1}^{q}\lambda_{r-1,i}(X)
	=
	\sum_{i=1}^{q}d_{r-1,i}^{\top}(X)
	\]
	if and only if
	\[
	q\ge \max\{\operatorname{rank}B_r(X),\Delta_{r-1}(X)\}.
	\]
	Thus, unlike the graph case, equality can occur only after both sequences
	have exhausted all their nonzero terms. As consequences, equality in the
	first partial sum and equality between the entire sequences are both equivalent
	to $X$ containing a unique $r$-simplex. The proof is based on the local
	down-Laplacian decomposition and the equality case of the Ky Fan inequality.
\end{abstract}
	\medskip
	
	
	\medskip
	
	
	\section{Introduction}\label{sc-1}
	
	The relation between the Laplacian spectrum of a graph and its degree sequence is a central topic in spectral graph theory. Let $G$ be a simple graph, let
	\[\lambda_0(G)=\bigl(\lambda_{0,1}(G),	\lambda_{0,2}(G),\ldots\bigr)	\]
	be its Laplacian spectrum arranged in weakly decreasing order, and let $d_0^{\top}(G)$ denote the conjugate degree sequence of
	$G$. In 1994, Grone and Merris~\cite{GM1994} conjectured that $	\lambda_0(G)\preceq d_0^{\top}(G)$, and this conjecture was proved by Bai~\cite{B2011}. The resulting statement is now known as the Grone--Merris--Bai theorem.
	
	The equality problem has also attracted considerable attention. In 1994, Merris~\cite{M1994} proved that $\lambda_0(G)=d_0^{\top}(G)$
	if and only if $G$ is a threshold graph. More recently, Cai, Chen, Yang, and Zhang~\cite{CCYZ2026} completely characterized
	all pairs $(G,q)$ for which
	\[\sum_{i=1}^{q}\lambda_{0,i}(G)=\sum_{i=1}^{q}d_{0,i}^{\top}(G).
	\]
	Their result shows that, in the graph case, equality may occur
	at a genuine intermediate partial sum and gives a complete
	description of the corresponding graph structures.
	
	Higher-dimensional analogues of graph Laplacians arise naturally on simplicial complexes. Let $X$ be an $r$-dimensional simplicial complex with $r\ge 1$, and let
	\[\lambda_{r-1}(X)=\bigl(\lambda_{r-1,1}(X),\lambda_{r-1,2}(X),\ldots\bigr)\]
	denote the eigenvalue sequence of $L_{r-1}^{\operatorname{up}}(X)$, arranged in weakly decreasing order. Let
	\[d_{r-1}^{\top}(X)=\bigl(\deg_{r-1,1}^{\top}(X),\deg_{r-1,2}^{\top}(X),\ldots\bigr)\]
	denote the conjugate $(r-1)$-degree sequence of $X$. The precise definitions of the majorization relation, the simplicial up-Laplacian,
	and the corresponding conjugate degree sequences are given in Section~\ref{sc-2}. In 2002, Duval and Reiner~\cite{DR2002} proposed a higher-dimensional analogue of the Grone--Merris conjecture involving the conjugate vertex-degree sequence. Very recently, Zhang, Song, and Fan~\cite{ZSF2026} showed that this conjecture fails in general. Independently, Huang~\cite{H2026} constructed further counterexamples, proved the second partial-sum inequality in the vertex-degree formulation, and classified its equality cases. Gupta~\cite{G2026} subsequently characterized, for pure complexes, equality between the entire up-Laplacian spectrum and the conjugate vertex-degree partition, and also studied bounds for the first two up-Laplacian eigenvalues. These results concern degrees of vertices. In contrast, the present paper concerns degrees of $(r-1)$-faces and the face-degree majorization theorem proved by Zhang, Song, and Fan:
	\[
	\lambda_{r-1}(X)\preceq d_{r-1}^{\top}(X).
	\]
	When $r=1$, this reduces exactly to the Grone--Merris--Bai
	theorem.
	
	The purpose of the present paper is to determine all equality cases in the partial-sum inequalities of this higher-dimensional majorization theorem. Define
\[\Delta_{r-1}(X)=\max_{\tau\in X(r-1)}\deg_{r-1}(X;\tau).\]
Our main result is stated as follows.
\begin{theorem}\label{thm-main}
	Let $X$ be an $r$-dimensional simplicial complex with $r\ge 2$, and let $q$ be a positive integer. Then
	\begin{equation}\label{eq-main-equality}
		\sum_{i=1}^{q}\lambda_{r-1,i}(X)=\sum_{i=1}^{q}d_{r-1,i}^{\top}(X)
	\end{equation}
	if and only if
	\begin{equation}\label{eq-main-condition}
		q\ge \max\left\{\operatorname{rank}B_r(X),\Delta_{r-1}(X)\right\}.
	\end{equation}
\end{theorem}
Theorem~\ref{thm-main} reveals a sharp distinction between the graph case and the higher-dimensional case. Indeed, $\operatorname{rank}B_r(X)$ is precisely the number of nonzero eigenvalues of $L_{r-1}^{\operatorname{up}}(X)$, whereas $\Delta_{r-1}(X)$ is precisely the number of nonzero entries of $d_{r-1}^{\top}(X)$. Hence equality in \eqref{eq-main-equality} occurs exactly when both sides have
already exhausted all their nonzero terms.
	
As an immediate consequence of Theorem~\ref{thm-main}, we first obtain a characterization of equality in the first partial-sum inequality.
\begin{corollary}\label{cor-first}
	Let $X$ be an $r$-dimensional simplicial complex with $r\ge 2$.
	Then
	\[\lambda_{r-1,1}(X)=d_{r-1,1}^{\top}(X)\] if and only if $f_r(X)=1$.
\end{corollary}
The equality case for the first partial sum already determines equality between the entire sequences.
\begin{corollary}\label{cor-full}
	Let $X$ be an $r$-dimensional simplicial complex with $r\ge 2$.
	Then
	\[\lambda_{r-1}(X)=d_{r-1}^{\top}(X)\]
	if and only if $f_r(X)=1$. In particular, if $X$ is pure, then $\lambda_{r-1}(X)=d_{r-1}^{\top}(X)$	if and only if $X\cong\Delta_r$, where $\Delta_r$ denotes the
	$r$-dimensional simplex.
\end{corollary}

Finally, the rank condition in Theorem~\ref{thm-main} can be expressed in homological terms, yielding the following reformulation.
\begin{corollary}\label{cor-homology}
	Let $X$ be an $r$-dimensional simplicial complex with $r\ge 2$, and let $\beta_r(X)=\dim \widetilde{H}_r(X;\mathbb{R})$.
	Then
	\[\sum_{i=1}^{q}\lambda_{r-1,i}(X)=\sum_{i=1}^{q}d_{r-1,i}^{\top}(X)
	\]
	if and only if
	\[
	q\ge 
	\max\left\{
	f_r(X)-\beta_r(X),
	\Delta_{r-1}(X)
	\right\}.
	\]
\end{corollary}

The following example illustrates that the equality threshold in
Theorem~\ref{thm-main} can be strictly smaller than the number of
$r$-simplices, even though equality is never a genuine intermediate
partial-sum equality.
\begin{example}\label{ex-simplex-boundary}
	Let $X=\partial\Delta_{r+1}$ be the boundary complex of the
	$(r+1)$-dimensional simplex. Then $X$ is an $r$-dimensional simplicial
	complex with
	\[
	 f_r(X)=r+2,\qquad \beta_r(X)=1,
	 \qquad\text{and}\qquad \Delta_{r-1}(X)=2.
	\]
	Indeed, $X$ triangulates the $r$-sphere, and every $(r-1)$-simplex is
	contained in exactly two $r$-simplices. Hence
	\[
	 \operatorname{rank}B_r(X)=f_r(X)-\beta_r(X)=r+1.
	\]
	Theorem~\ref{thm-main} therefore gives
	\[
	 \sum_{i=1}^{q}\lambda_{r-1,i}(X)
	 =\sum_{i=1}^{q}d_{r-1,i}^{\top}(X)
	 \quad\Longleftrightarrow\quad q\ge r+1.
	\]
	In particular, equality first occurs at $q=r+1<f_r(X)=r+2$;
	nevertheless, all nonzero eigenvalues have already been included.
\end{example}

The proof of Theorem~\ref{thm-main} is based on the local down-Laplacian decomposition introduced by Zhang, Song, and Fan~\cite{ZSF2026}. Merely applying equality results to the individual link graphs is not sufficient, because equality must be compatible across all links. The main new ingredient is therefore a simultaneous analysis of the equality case in the Ky Fan inequality. Global equality produces a common maximizing subspace for all local down-Laplacians. Its coordinate support must then be compatible with the families of $r$-simplices containing the various $(r-2)$-simplices. We show that this compatibility is rigid when $r\ge 2$: it forces every local matrix to have rank at most $q$, and consequently forces $\operatorname{rank}B_r(X)\le q$ and $\Delta_{r-1}(X)\le q$. Thus equality in the $q$-th partial-sum inequality can occur only when both the up-Laplacian spectrum and the conjugate $(r-1)$-degree sequence have exhausted all their nonzero terms by the $q$-th position.

The remainder of the paper is organized as follows. Section~\ref{sc-2} introduces the necessary notation and preliminary results concerning
simplicial Laplacians, face degrees, majorization, and local down-Laplacians. In Section~\ref{sc-3}, we establish the equality conditions for the Ky Fan inequality and the local majorization inequalities that will be used in the proof. Section~\ref{sc-4} is devoted to the proof of Theorem~\ref{thm-main} and its corollaries.
	
	\section{Preliminaries}\label{sc-2}
Let $V$ be a finite set. An \textit{abstract simplicial complex} (or simply \textit{complex}) $X$ on $V$ is a collection of the subsets of $V$ closed under taking subsets; that is, if $F\in X$ and $F'\subseteq F$, then $F'\in X$.  An \textit{$k$-simplex} (or \textit{$k$-face}) in $X$ is an element of $X$ with cardinality $k+1$, and its \textit{dimension} is defined as its cardinality minus one. The \textit{dimension} of $X$, denoted by $\dim X$, is the maximum dimension of all simplices of $X$. Two abstract simplicial complexes $X_1$ and $X_2$ are said to be \textit{isomorphic}, denoted by $X_1\cong X_2$, if there is a bijective correspondence $f$ mapping the vertex set of $X_1$ to the vertex set of $X_2$ such that $\left\{v_0, v_1, \dots, v_k\right\} \in X_1$ if and only if $\left\{f\left(v_0\right), f\left(v_1\right), \dots, f\left(v_k\right)\right\} \in X_2$ for any $k$. 

Let $X(k)$ be the set of all $k$-simplices of $X$ for $k \ge 0$, and write $f_k(X) \colon= |X(k)|$ for its cardinality. For each $k$-simplex $\sigma$, define $N^{+}_X(\sigma)\colon=\left\{\tau\in X(k+1)\colon\sigma\subseteq\tau\right\}$, 
called the \textit{upper neighborhood} of $\sigma$; its cardinality is referred to as the \textit{upper degree} of $\sigma$, denoted by $\deg^{+}_X(\sigma)$. The \textit{link} of $\sigma$ is defined by $\operatorname{lk}_X\sigma\colon= \{\tau\in X\colon \tau\cap\sigma=\emptyset,\tau\cup\sigma\in X\}$.	
	
A simplex $\sigma$ is said to be \textit{oriented} if its vertices are assigned a linear ordering, denoted by $[\sigma]$.  Throughout the paper, all simplices are assumed to be oriented. We assume that $\emptyset\in X$, where $\emptyset$ is called the \emph{empty simplex} and has dimension $-1$. Then $X(-1)=\{\emptyset\}$ and $f_{-1}(X)=1$. For $-1\le k\le \dim X$, the \textit{$k$-chain group} $C_k(X,\mathbb{R})$ is the real vector space formally generated by all oriented $k$-simplices as a basis. For each integer $0\le  k\le \dim X$, the \emph{boundary map} $\partial_k(X)\colon C_k(X, \mathbb{R}) \rightarrow C_{k-1}(X,\mathbb{R})$ is the linear operator, defined on a $k$-simplex
$\sigma=[v_0,\ldots,v_k]$ by
\[\partial_k(X)\sigma=\sum_{j=0}^{k}(-1)^j[v_0,\ldots,\widehat{v_j},\ldots,v_k],\]
where $\widehat{v_j}$ indicates that the vertex $v_j$ is omitted. Equivalently,
\[\partial_k(X)\sigma=\sum_{\eta\in X(k-1)}(\sigma:\eta)\eta,\]
where the \textit{incidence sign} $(\sigma:\eta)$ is defined by
\[(\sigma:\eta)=\begin{cases}1,& \text{if $\eta\subset\sigma$ and the orientation of $\eta$ agrees with that induced by $\sigma$,}\\-1,& \text{if $\eta\subset\sigma$ and the orientation of $\eta$ is opposite to that induced by $\sigma$,}\\
	0,& \text{if $\eta\not\subset\sigma$.}\end{cases}
\]
The boundary map is extended linearly to all chains in $C_k(X,\mathbb{R})$. With respect to the standard inner product, let $\partial_k^*(X)\colon C_{k-1}(X,\mathbb{R})\to C_k(X,\mathbb{R})$ denote the adjoint of $\partial_k(X)$. The $k$-dimensional
\emph{up-Laplacian} and \emph{down-Laplacian} are defined by \[
L_k^{\operatorname{up}}(X)=\partial_{k+1}(X)\partial_{k+1}(X)^*\quad\text{and}\quad L_k^{\operatorname{down}}(X)=\partial_k(X)^*\partial_k(X),
\]
respectively, which have been extensively studied in \cite{DR2002,FWW2024,ZHL2026,L2020,SY2020,KRS2000}.
It is well known that $\partial_k(X)\circ\partial_{k+1}(X)=0$
for every $0\le  k\le \dim X-1$. The \textit{$k$-th reduced homology group} of $X$ is given by \[\widetilde{H}_k(X;\mathbb{R})\colon=\ker\partial_k(X) / \mathrm{im} \partial_{k+1}(X).\]

For each integer $0\le  k\le  \dim X$, let $B_k(X)=\bigl(B_k(X)_{\eta,\sigma}\bigr)$ denote the matrix representation of $\partial_k(X)$ with respect to the standard bases of $C_k(X,\mathbb{R})$ and $C_{k-1}(X,\mathbb{R})$. Its rows and columns are indexed by $X(k-1)$ and $X(k)$, respectively, and
	its entries are given by
	\[B_k(X)_{\eta,\sigma}=
	\begin{cases}
		(\sigma:\eta), & \text{if }\eta\subset\sigma,\\
		0, & \text{otherwise}.
	\end{cases}
	\]
	Thus, with respect to the standard basis of $C_k(X,\mathbb{R})$, the matrix representations of the $k$-th up-Laplacian and down-Laplacian are
	\begin{equation}\label{eq-2}
		L_k^{\operatorname{up}}(X)=B_{k+1}(X)B_{k+1}(X)^{\top},\qquad 0\le  k\le  r-1
	\end{equation}
	and
	\begin{equation}\label{eq-6}
		L_k^{\operatorname{down}}(X)=B_k(X)^{\top}B_k(X),\qquad 0\le  k\le  r,
	\end{equation}
	respectively. Let $\mathbf{s}_k^{\operatorname{up}}(X)$ and $\mathbf{s}_k^{\operatorname{down}}(X)$ denote the spectra of $L_k^{\operatorname{up}}(X)$ and $L_k^{\operatorname{down}}(X)$, respectively. Both $L_k^{\operatorname{up}}(X)$ and
	$L_k^{\operatorname{down}}(X)$ are real symmetric positive semidefinite matrices, and hence all their eigenvalues are real and
	nonnegative. Moreover, it is a standard fact in linear algebra that, for any real matrix $M$, the matrices $M^{\top}M$ and $MM^{\top}$
	have the same nonzero eigenvalues, including multiplicities. Therefore, by \eqref{eq-2} and \eqref{eq-6}, 
\begin{equation}\label{eq-01}
	\mathbf{s}_k^{\operatorname{down}}\left(X\right) \stackrel{\circ}{=} \mathbf{s}_{k-1}^{\operatorname{up}}\left(X\right),
\end{equation}
where $\stackrel{\circ}{=}$ denotes equality after zero eigenvalues are disregarded. The entries of these matrices admit the following explicit descriptions.
\begin{theorem}[{\cite{HJ2013}}]\label{thm-15}
	Let $X$ be a simplicial complex. For each $-1\le  k\le \dim X-1$, the rows and columns of $L_k^{\operatorname{up}}(X)$ are indexed by the $k$-simplices of $X$, and, for all $\sigma,\tau\in X(k)$,
	\[L_k^{\operatorname{up}}(X)_{\sigma,\tau}=
	\begin{cases}
		\deg_X^{+}(\sigma),
		& \text{if }\sigma=\tau,\\[2mm]
		(\sigma\cup\tau:\sigma)(\sigma\cup\tau:\tau),
		& \text{if }|\sigma\cap\tau|=k
		\text{ and }\sigma\cup\tau\in X,\\[2mm]
		0,& \text{otherwise}.
	\end{cases}
	\]
\end{theorem}
\begin{theorem}[{\cite{HJ2013}}]\label{thm-13}
	Let $X$ be a simplicial complex. For every integer $0\le  k\le \dim X$, the rows and columns of $L_k^{\operatorname{down}}(X)$ are indexed by the $k$-simplices of $X$, and, for all $\sigma,\tau\in X(k)$,
	\[L_k^{\operatorname{down}}(X)_{\sigma,\tau}=
	\begin{cases}k+1,& \text{if }\sigma=\tau,\\[2mm]
		(\sigma:\sigma\cap\tau)(\tau:\sigma\cap\tau),& \text{if }|\sigma\cap\tau|=k,\\[2mm]
		0,& \text{otherwise}.\end{cases}\]
\end{theorem}

Let $\dim X=r$. For $-1\le k\le r-1$, let $\lambda_{k}(X)=\bigl(\lambda_{k,1}(X),\lambda_{k,2}(X),\ldots\bigr)$ denote the eigenvalue sequence of $L_{k}^{\operatorname{up}}(X)$, arranged in weakly decreasing order. For any $\sigma\in X(k)$, the \emph{degree} of $\sigma$ is defined by $\deg_k(X;\sigma)=\left|\left\{F\in X(r):\sigma\subseteq F\right\}\right|.$
The $k$-degree sequence of $X$, denoted by
\[d_k(X)=\bigl(\deg_{k,1}(X),\cdots,\deg_{k,f_k(X)}(X)\bigr),\]
	is the sequence of degrees of all $k$-simplices of $X$ arranged in nonincreasing order. The \emph{conjugate $k$-degree sequence} of $X$, denoted by
	\[d_k^{\top}(X)=\bigl(\deg_{k,1}^{\top}(X),\deg_{k,2}^{\top}(X),\cdots\bigr),\]
	is the conjugate of $d_k(X)$, defined by $\deg_{k,j}^{\top}(X):=\left|\left\{t:\deg_{k,t}(X)\ge  j\right\}\right|$ for $j\ge  1$. 
When comparing two sequences, we pad them with zeros whenever necessary. Let $\mathbf{x}=(x_1,x_2,\ldots)$ and $\mathbf{y}=(y_1,y_2,\ldots)$ be two nonnegative, nonincreasing, finitely supported sequences. We say that $\mathbf{x}$ is \emph{majorized} by
$\mathbf{y}$, denoted by $\mathbf{x}\preceq\mathbf{y}$, if $\sum_{i=1}^qx_i\le \sum_{i=1}^qy_i$ for $q\ge 1$ and $\sum_{i\ge 1}x_i=\sum_{i\ge 1}y_i$. 	

We next recall the local down-Laplacian construction of Zhang, Song, and Fan~\cite{ZSF2026}. Let $r\ge 2$. For each $\eta\in X(r-2)$, define the local boundary matrix
$B_r(X;\eta)$ by
\[\bigl(B_r(X;\eta)\bigr)_{\tau,\sigma}=\begin{cases}(\sigma:\tau),
	& \text{if }\eta\subset\tau\subset\sigma,\\
	0,	& \text{otherwise},\end{cases}\]
where $\tau\in X(r-1)$ and $\sigma\in X(r)$. Let $A_\eta=B_r(X;\eta)^{\top}B_r(X;\eta)$ and $H_\eta=\operatorname{lk}_X\eta.$
Since $\eta$ has dimension $r-2$ and $X$ has dimension $r$, $H_\eta$ is a graph. For every $v\in V(H_\eta)$, it holds
\[\deg_{H_\eta}(v)=\deg_{r-1}\bigl(X;\eta\cup\{v\}\bigr).\] 
For later use, let $\mathcal{F}_\eta=\{\sigma\in X(r):\eta\subset\sigma\}$. Moreover, we have
	\begin{equation}\label{eq-local-diagonal}
		(A_\eta)_{\sigma,\sigma}=2\qquad\text{for every }\sigma\in\mathcal{F}_\eta,
	\end{equation}
	since exactly two $(r-1)$-faces lie strictly between $\eta$ and $\sigma$. The following two lemmas are due to Zhang, Song, and
	Fan~\cite{ZSF2026}.
	\begin{lemma}[{\cite[Lemma~3.1]{ZSF2026}}]\label{lem-1}
		Let $X$ be an $r$-dimensional simplicial complex with $r\ge 1$.	For every $\eta\in X(r-2)$, there exists a signed permutation matrix
		$P_\eta$ such that
		\[P_\eta^{\top}A_\eta P_\eta=	L_1^{\operatorname{down}}(H_\eta)\oplus O.\]
	\end{lemma}
	\begin{lemma}[{\cite[Lemma~3.2]{ZSF2026}}]\label{lem-2}
		Let $X$ be an $r$-dimensional simplicial complex with $r\ge 1$. Then
		\[rL_r^{\operatorname{down}}(X)=\sum_{\eta\in X(r-2)}A_\eta.\]
	\end{lemma}

	\section{Equality conditions in the majorization inequality}\label{sc-3}
	
	For a real symmetric matrix $M$ of order $m$, let $\lambda_1(M)\ge \lambda_2(M)\ge \cdots\ge \lambda_m(M)$ be all eigenvalues of $M$. For every positive integer $q$, define
	\[\Phi_q(M)=\sum_{i=1}^{q}\lambda_i(M),\]
	where we set $\lambda_i(M)=0$ for $i>m$.  For $1\le  q\le  m$, a $q$-dimensional subspace $U\subseteq\mathbb{R}^m$ is called a \emph{maximizing subspace} for $\Phi_q(M)$ if there exists a matrix $V\in\mathbb{R}^{m\times q}$ whose columns form an orthonormal basis of $U$ such that $\operatorname{tr}\left(V^{\top}MV\right)=\Phi_q(M)$. 
	
	We shall repeatedly use the following Ky Fan maximum principle. 
	\begin{lemma}[Ky Fan's maximum principle {\cite[Corollary 4.3.39]{HJ2012}}]\label{lem-ky-fan}
		Let $M$ be a real symmetric matrix of order $m$. Then, for every
		integer $1\le  q\le  m$,\[\Phi_q(M)=\max_{\substack{V\in\mathbb{R}^{m\times q}\\[1mm] V^{\top}V=I_q}}\operatorname{tr}\left(V^{\top}MV\right).\]
	\end{lemma}
By repeated application of the classical Ky Fan inequality \cite[Proposition~A.6]{MOA2011}, we obtain the following finite-sum
form.
\begin{lemma}[Ky Fan's inequality {\cite[Proposition~A.6]{MOA2011}}]\label{lem-kyfan-inequality}
	Let $M_1,\ldots,M_s\in\mathbb{R}^{m\times m}$ be symmetric matrices, and let $1\le  q\le  m$. Then
\[\Phi_q\left(\sum_{j=1}^{s}M_j\right)\le  \sum_{j=1}^{s}\Phi_q(M_j).\]
\end{lemma}

The equality characterization for the Ky Fan maximum principle is standard; see, for example, the discussion in \cite{F2015}. We shall use the following consequence concerning simultaneous equality in the Ky Fan subadditivity inequality for positive semidefinite matrices. For
completeness, we provide a proof.
\begin{lemma}\label{lem-common-subspace}
	Let $M_1,\ldots,M_s\in\mathbb{R}^{m\times m}$ be symmetric positive semidefinite matrices, and let $1\le  q<m$. Suppose that
	\begin{equation}\label{eq-common-equality}
		\Phi_q\left(\sum_{j=1}^{s}M_j\right)=\sum_{j=1}^{s}\Phi_q(M_j).
	\end{equation}
	Then there exists a $q$-dimensional subspace
	$U\subseteq\mathbb{R}^m$ which is a maximizing subspace for
	$\Phi_q(M_j)$ for every $j$. Moreover,
	\begin{enumerate}[label=\emph{(\roman*)}]
		\item
		if $\operatorname{rank}M_j\le  q$, then
		$\operatorname{im}M_j\subseteq U$;
		\item
		if $\operatorname{rank}M_j>q$, then
		$U\subseteq\operatorname{im}M_j$.
	\end{enumerate}
\end{lemma}
\begin{proof}
	By Lemma~\ref{lem-ky-fan}, there exists a matrix
	$V\in\mathbb{R}^{m\times q}$ satisfying $V^{\top}V=I_q$ such that
	\[\operatorname{tr}\left(V^{\top}\left(\sum_{j=1}^{s}M_j\right)V\right)=\Phi_q\left(\sum_{j=1}^{s}M_j\right).\]
Let $U=\operatorname{im}V$. By \eqref{eq-common-equality}, we have
\begin{equation}\label{eq-1}
	\begin{aligned}
		\sum_{j=1}^{s}\left(\Phi_q(M_j)-\operatorname{tr}\left(V^{\top}M_jV\right)\right)&=	\sum_{j=1}^{s}\Phi_q(M_j)-\sum_{j=1}^{s}\operatorname{tr}\left(V^{\top}M_jV\right)\\&=\Phi_q\left(\sum_{j=1}^{s}M_j\right)-\operatorname{tr}\left(V^{\top}\left(\sum_{j=1}^{s}M_j\right)V\right)=0.
	\end{aligned}
\end{equation}
	It follows from  Lemma~\ref{lem-ky-fan} that $\Phi_q(M_j)\ge \operatorname{tr}\left(V^{\top}M_jV\right)$ for all $j$. Combining this with \eqref{eq-1}, we obtain $\operatorname{tr}\left(V^{\top}M_jV\right)=\Phi_q(M_j)$ for all $j$. This implies that $U$ is a maximizing $q$-subspace for every $M_j$.
	
	We next prove (i). Suppose that $\operatorname{rank}M_j\le  q$. Since $M_j$ is positive semidefinite, all its eigenvalues are
nonnegative, and exactly $\operatorname{rank}M_j$ of them are positive. Hence, as $\operatorname{rank}M_j\le  q$, the first $q$ eigenvalues contain all the nonzero eigenvalues of $M_j$. Therefore,
\begin{equation}\label{eq-3}
	\Phi_q(M_j)=\sum_{i=1}^{q}\lambda_i(M_j)=\sum_{i=1}^{m}\lambda_i(M_j)=\operatorname{tr}M_j.
\end{equation}
Choose $W\in\mathbb{R}^{m\times(m-q)}$ such that $[V\ W]$ is an orthogonal matrix. Then $U^\perp=\operatorname{im}W$ and $VV^{\top}+WW^{\top}=I_m.$
	Consequently,
	\[
	\begin{aligned}
	 \operatorname{tr}M_j
	 &=\operatorname{tr}\left(M_jVV^{\top}\right)
	   +\operatorname{tr}\left(M_jWW^{\top}\right)\\
	 &=\operatorname{tr}\left(V^{\top}M_jV\right)
	   +\operatorname{tr}\left(W^{\top}M_jW\right)\\
	 &=\Phi_q(M_j)+\operatorname{tr}\left(W^{\top}M_jW\right).
	\end{aligned}
	\]
	It follows from \eqref{eq-3} that $	\operatorname{tr}\left(W^{\top}M_jW\right)=0$.
	Since $M_j$ is positive semidefinite, it has a unique positive
	semidefinite square root $M_j^{1/2}$. Hence, we have
	\[\operatorname{tr}\left(\left(M_j^{1/2}W\right)^{\top}M_j^{1/2}W\right)=\operatorname{tr}\left(W^{\top}M_jW\right)=0,
	\]
	and thus $	M_j^{1/2}W=0$. This implies that $U^\perp=\operatorname{im}W\subseteq\ker M_j$. Since $M_j$ is symmetric, it follows that
$\operatorname{im}M_j=(\ker M_j)^{\perp}\subseteq U$.

Now we prove (ii). Suppose that $\operatorname{rank}M_j>q$. Let $\mathbf{x}_1,\ldots,\mathbf{x}_m$ be an orthonormal eigenbasis
of $M_j$ corresponding to the eigenvalues $\lambda_1(M_j),\ldots,\lambda_m(M_j)$, respectively. Then $\sum_{i=1}^{m}\mathbf{x}_i\mathbf{x}_i^{\top}=I_m$, and thus \begin{equation}\label{eq-7}
	M_j=M_j\sum_{i=1}^{m}\mathbf{x}_i\mathbf{x}_i^{\top}=\sum_{i=1}^{m}\lambda_i(M_j)\mathbf{x}_i\mathbf{x}_i^{\top}.
\end{equation}
Write $\rho=\operatorname{rank}M_j$. Since $M_j$ is positive semidefinite, we have
\[\lambda_1(M_j)\ge \cdots\ge \lambda_\rho(M_j)>0=\lambda_{\rho+1}(M_j)=\cdots=\lambda_m(M_j).\]
Since $\rho>q$, we have $\lambda_q(M_j)>0$. Write $V=[\mathbf{v}_1,\ldots,\mathbf{v}_q]$ and $W=[\mathbf{w}_1,\ldots,\mathbf{w}_{m-q}]$.
Since $[V\ W]$ is an orthogonal matrix, the vectors $\mathbf{v}_1,\ldots,\mathbf{v}_q, \mathbf{w}_1,\ldots,\mathbf{w}_{m-q}$ are orthonormal. For 
$1\le  i\le  m$, we have \[\mathbf{x}_i=\sum_{\ell=1}^{q}
\left(\mathbf{v}_{\ell}^{\top}\mathbf{x}_i\right)\mathbf{v}_{\ell}+\sum_{k=1}^{m-q}\left(\mathbf{w}_{k}^{\top}\mathbf{x}_i\right)\mathbf{w}_{k}.\]
Since $\mathbf{x}_i$ is a unit vector, we have \begin{equation}\label{eq-4}
	1=\mathbf{x}_i^{\top}\mathbf{x}_i=\sum_{\ell=1}^{q}\left(\mathbf{v}_{\ell}^{\top}\mathbf{x}_i\right)^2+\sum_{k=1}^{m-q}\left(\mathbf{w}_{k}^{\top}\mathbf{x}_i\right)^2.
\end{equation}
Let $p_i=\sum_{\ell=1}^{q}\left(\mathbf{v}_{\ell}^{\top}\mathbf{x}_i\right)^2$. By \eqref{eq-4}, we have $0\le  p_i\le 1$. Moreover, by \eqref{eq-7}, we obtain
{\small \begin{equation}\label{eq-8}
	\sum_{i=1}^{m}p_i=\sum_{i=1}^{m}\sum_{\ell=1}^{q}\left(\mathbf{v}_{\ell}^{\top}\mathbf{x}_i\right)^2=\sum_{i=1}^{m}\sum_{\ell=1}^{q}\mathbf{v}_{\ell}^{\top}\mathbf{x}_i\mathbf{x}_i^\top\mathbf{v}_{\ell}=\sum_{\ell=1}^{q}\mathbf{v}_{\ell}^{\top}\left(\sum_{i=1}^{m}\mathbf{x}_i\mathbf{x}_i^{\top}\right)\mathbf{v}_{\ell}=\sum_{\ell=1}^{q}\mathbf{v}_{\ell}^{\top}\mathbf{v}_{\ell}=q.
\end{equation}}
Therefore, by \eqref{eq-7} again,
\begin{equation}\label{eq-5}
	\begin{aligned}
		\operatorname{tr}\left(V^{\top}M_jV\right)&=\sum_{\ell=1}^{q}\mathbf{v}_{\ell}^{\top}M_j\mathbf{v}_{\ell}=\sum_{\ell=1}^{q}\sum_{i=1}^{m}\lambda_i(M_j)\left(\mathbf{v}_{\ell}^{\top}\mathbf{x}_i\right)^2=\sum_{i=1}^{m}\lambda_i(M_j)p_i.
	\end{aligned}
\end{equation}
\begin{claim}\label{c-1}
	For every $i>\rho$, $p_i=0$.
	\end{claim}
\begin{proof}
	Let $\delta=\sum_{i=\rho+1}^{m}p_i$. Suppose, to the contrary, that there exists $i>\rho$ such that $p_i>0$. Since $p_i\ge 0$ for any $i$, we have
$\delta>0$. Since $\lambda_i(M_j)=0$ for $i>\rho$, it follows from
\eqref{eq-5} that
\[\operatorname{tr}\left(V^{\top}M_jV\right)=\sum_{i=1}^{\rho}\lambda_i(M_j)p_i.\]
Combining this with \eqref{eq-8}, we obtain
\[\begin{aligned}
	\Phi_q(M_j)-\operatorname{tr}\left(V^{\top}M_jV\right)&=\sum_{i=1}^{q}\lambda_i(M_j)-\sum_{i=1}^{\rho}\lambda_i(M_j)p_i=
	\sum_{i=1}^{q}\lambda_i(M_j)(1-p_i)-\sum_{i=q+1}^{\rho}\lambda_i(M_j)p_i\\&\ge \lambda_q(M_j)\left(\sum_{i=1}^{q}(1-p_i)-\sum_{i=q+1}^{\rho}p_i\right)=\lambda_q(M_j)\left(q-\sum_{i=1}^{\rho}p_i\right)\\&=\lambda_q(M_j)\left(\sum_{i=1}^{m}p_i-\sum_{i=1}^{\rho}p_i\right)=\delta \lambda_q(M_j).
\end{aligned}
\]
By noticing that $\delta>0$ and $\lambda_q(M_j)>0$, we have
\[\Phi_q(M_j)-\operatorname{tr}\left(V^{\top}M_jV\right)\ge \lambda_q(M_j)\delta>0,\]
which contradicts $\operatorname{tr}\left(V^{\top}M_jV\right)=\Phi_q(M_j)$.
\end{proof}
By the definition of $p_i$ and Claim \ref{c-1}, for every $i>\rho$, we have $\sum_{\ell=1}^{q}\left(\mathbf{v}_{\ell}^{\top}\mathbf{x}_i\right)^2=0$. This implies that $\mathbf{v}_{\ell}^{\top}\mathbf{x}_i=0$ for all $1\le \ell\le  q$ and $\rho+1\le  i\le  m$.
Since $\ker M_j=\operatorname{span}\left\{\mathbf{x}_{\rho+1},\ldots,\mathbf{x}_m\right\}$, each $\mathbf{v}_{\ell}$ is orthogonal to $\ker M_j$. Therefore, we have \[\mathbf{v}_{\ell}\in(\ker M_j)^{\perp}=\operatorname{im}M_j \] for $1\le \ell\le  q$. Since $U=\operatorname{im}V=\operatorname{span}\left\{\mathbf{v}_1,\ldots,\mathbf{v}_q\right\}$,
we conclude that $U\subseteq\operatorname{im}M_j$. This proves (ii).
\end{proof}

We next state explicitly the local inequality used in the proof of the
face-degree majorization theorem. This formulation will make the equality
argument below self-contained.
\begin{lemma}\label{lem-local-majorization}
	Let $X$ be an $r$-dimensional simplicial complex with $r\ge 2$.
	For every $\eta\in X(r-2)$ and every positive integer $q$,
	\begin{equation}\label{eq-local-majorization}
		\Phi_q(A_\eta)
		\le
		\sum_{\substack{\tau\in X(r-1)\\ \eta\subset\tau}}
		\min\{\deg_{r-1}(X;\tau),q\}.
	\end{equation}
\end{lemma}
\begin{proof}
	By Lemma~\ref{lem-1}, $A_\eta$ is orthogonally similar to
	$L_1^{\operatorname{down}}(H_\eta)\oplus O$. Moreover,
	$L_1^{\operatorname{down}}(H_\eta)$ and
	$L_0^{\operatorname{up}}(H_\eta)$ have the same nonzero eigenvalues,
	including multiplicities. Consequently,
	\[
	 \Phi_q(A_\eta)=\Phi_q\bigl(L_0^{\operatorname{up}}(H_\eta)\bigr).
	\]
	Applying the Grone--Merris--Bai theorem~\cite{GM1994,B2011} to the
	link graph $H_\eta$ gives
	\[
	 \Phi_q\bigl(L_0^{\operatorname{up}}(H_\eta)\bigr)
	 \le \sum_{v\in V(H_\eta)}\min\{\deg_{H_\eta}(v),q\}.
	\]
	The vertices $v$ of $H_\eta$ are in bijection with the
	$(r-1)$-simplices $\tau=\eta\cup\{v\}$ containing $\eta$, and
	\[
	 \deg_{H_\eta}(v)=\deg_{r-1}(X;\eta\cup\{v\}).
	\]
	Substitution yields \eqref{eq-local-majorization}.
\end{proof}

The proof of the face-degree majorization theorem of Zhang, Song, and Fan~\cite{ZSF2026} combines the global Ky Fan inequality with Lemma~\ref{lem-local-majorization}. Although the equality cases were not studied there, equality in the final inequality of the majorization chain forces equality in both the global Ky Fan step and every local majorization step. We record these equality consequences in the following proposition.
\begin{proposition}\label{prop-equality-chain}
	Let $X$ be an $r$-dimensional simplicial complex with $r\ge 2$, and let $q$ be a positive integer. Suppose that
	\begin{equation}\label{eq-partial-assumption}
		\sum_{i=1}^{q}\lambda_{r-1,i}(X)=\sum_{i=1}^{q}d_{r-1,i}^{\top}(X).
	\end{equation}
	Then
	\begin{equation}\label{eq-global-local-equality}
		\Phi_q\left(\sum_{\eta\in X(r-2)}A_\eta\right)=\sum_{\eta\in X(r-2)}\Phi_q(A_\eta),
	\end{equation}
	and, for every $\eta\in X(r-2)$,
	\begin{equation}\label{eq-every-local-equality}
		\Phi_q(A_\eta)=\sum_{\tau\in X(r-1)\colon\eta\subset\tau}\min\left\{\deg_{r-1}(X;\tau),q\right\}.
	\end{equation}
\end{proposition}
\begin{proof}
	The proof of the face-degree majorization theorem in
	Zhang, Song, and Fan~\cite{ZSF2026} is based on the following chain
	of inequalities:
	\[\begin{aligned}
		\sum_{i=1}^{q}\lambda_{r-1,i}(X)&=\frac{1}{r}\Phi_q\left(\sum_{\eta\in X(r-2)}A_\eta\right)\le \frac{1}{r}\sum_{\eta\in X(r-2)}\Phi_q(A_\eta)\\&\le  \frac{1}{r}\sum_{\eta\in X(r-2)}\sum_{\tau\in X(r-1)\colon \eta\subset\tau}\min\left\{\deg_{r-1}(X;\tau),q\right\}
		\\&=\sum_{\tau\in X(r-1)}\min\left\{\deg_{r-1}(X;\tau),q\right\}=\sum_{i=1}^{q}d_{r-1,i}^{\top}(X).
	\end{aligned}\]
Here, the first inequality is the Ky Fan inequality, and the second
	inequality follows from Lemma~\ref{lem-local-majorization}.
By \eqref{eq-partial-assumption}, the first and the last terms of the above chain are equal. Since all inequalities in the chain are in the same direction, every intermediate inequality must also be an equality. Equality in the first inequality gives
\[\frac{1}{r}\Phi_q\left(\sum_{\eta\in X(r-2)}A_\eta\right)=\frac{1}{r}\sum_{\eta\in X(r-2)}\Phi_q(A_\eta).\]
Multiplying both sides by $r$ yields
\eqref{eq-global-local-equality}.

	It remains to consider the equality in the local majorization
	inequalities. For every $\eta\in X(r-2)$, define
	\[\delta_\eta=\sum_{\tau\in X(r-1)\colon \eta\subset\tau}\min\left\{\deg_{r-1}(X;\tau),q\right\}-\Phi_q(A_\eta).\]
	By Lemma~\ref{lem-local-majorization}, we have $\delta_\eta\ge 0$
	for every $\eta\in X(r-2)$. Since equality holds in the second inequality of the above chain, we have $\sum_{\eta\in X(r-2)}\delta_\eta=0$, and thus $\delta_\eta=0$ for every $\eta\in X(r-2)$. Therefore, 	\eqref{eq-every-local-equality} holds.
\end{proof}
\section{The proofs of main results}\label{sc-4}

Before proceeding with the proofs, we introduce some notation and record a simple observation that will be used throughout this
section. We use $\mathbb R^{X(r)}$ to denote the $f_r(X)$-dimensional real vector space whose coordinates are indexed
by the $r$-simplices of $X$. For a subset $\mathcal F\subseteq X(r)$, let
\[\mathbb R^{\mathcal F}=\{\mathbf{x}\in\mathbb R^{X(r)}:\mathbf{x}_\sigma=0\text{ for all }\sigma\in X(r)\setminus\mathcal F\}.\]
For each $\sigma\in X(r)$, let $e_\sigma$ denote the standard basis vector of $\mathbb R^{X(r)}$ corresponding to the coordinate
$\sigma$. Notice that, for any family of subsets
$\{\mathcal F_i\}_{i\in I}$ of $X(r)$,
\begin{equation}\label{eq-coordinate-intersection}
	\bigcap_{i\in I}\mathbb R^{\mathcal F_i}
	=\mathbb R^{\cap_{i\in I}\mathcal F_i}.
\end{equation}
In particular, if a $q$-dimensional subspace is contained in every
$\mathbb R^{\mathcal F_i}$, then
$|\cap_{i\in I}\mathcal F_i|\ge q$. For every $\eta\in X(r-2)$, call $\eta$ \emph{active} if $\mathcal F_\eta\ne \emptyset$.  Since $A_\eta$ is supported on the coordinate set $\mathcal F_\eta$, we have
\begin{equation}\label{eq-15}
	\operatorname{im}A_\eta\subseteq\mathbb R^{\mathcal F_\eta}.
\end{equation}

We are now ready to prove Theorem~\ref{thm-main}.

\begin{proof}[Proof of Theorem~\ref{thm-main}]
Let $m=f_r(X)$. We first prove the sufficiency. Suppose that
\[
q\ge\max\{\operatorname{rank}B_r(X),\Delta_{r-1}(X)\}.
\]
Since $L_{r-1}^{\operatorname{up}}(X)=B_r(X)B_r(X)^\top$, it follows that
\[\operatorname{rank}L_{r-1}^{\operatorname{up}}(X)=\operatorname{rank}B_r(X)\le q\quad\text{and}\quad \Delta_{r-1}(X)\le q.\]
This implies $\lambda_{r-1,i}(X)=0$ for $i>q$. Therefore, we obtain
\[\sum_{i=1}^q\lambda_{r-1,i}(X)=\operatorname{tr}L_{r-1}^{\operatorname{up}}(X)=(r+1)f_r(X).\]
Moreover, since $\Delta_{r-1}(X)\le q$, we have $d_{r-1,i}^{\top}(X)=0$ for $i>q$. 
It follows that
\[\sum_{i=1}^qd_{r-1,i}^{\top}(X)=\sum_{\tau\in X(r-1)}\deg_{r-1}(X;\tau)=(r+1)f_r(X).\]
Consequently,
\[\sum_{i=1}^{q}\lambda_{r-1,i}(X)=\sum_{i=1}^{q}d_{r-1,i}^{\top}(X).\]

We now prove the necessity. Suppose that  \[\sum_{i=1}^{q}\lambda_{r-1,i}(X)=\sum_{i=1}^{q}d_{r-1,i}^{\top}(X).\]
If $q\ge  m$, then $\operatorname{rank}B_r(X)\le  m\le  q$ and $\Delta_{r-1}(X)\le  m\le  q$. Hence, the desired conclusion follows. Therefore, we may assume that $q<m$. By Proposition~\ref{prop-equality-chain}, we have\[	\Phi_q\left(\sum_{\eta\in X(r-2)}A_\eta\right)=\sum_{\eta\in X(r-2)}\Phi_q(A_\eta).\]
It follows from Lemma~\ref{lem-common-subspace} that there exists a $q$-dimensional subspace $U\subseteq \mathbb R^{m}$
which is a maximizing subspace for every $A_\eta$. We distinguish two cases according to the value of $q$.

\vspace{10pt}
\noindent\textbf{Case 1.} $q=1$. 
\vspace{10pt}

In this case, \eqref{eq-global-local-equality} reduces to \begin{equation}\label{eq-12}
	\lambda_1\left(\sum_{\eta\in X(r-2)}A_\eta\right)=\sum_{\eta\in X(r-2)}\lambda_1(A_\eta).
\end{equation}
Choose a unit eigenvector $\mathbf{x}\in\mathbb{R}^{f_r(X)}$ corresponding to $\lambda_1\left(\sum_{\eta\in X(r-2)}A_\eta\right)$, where the coordinates of $\mathbf{x}$ are indexed by the $r$-simplices of $X$. We first derive two consequences of this equality.
\begin{claim}\label{f-3}
	\rm	For every $\eta\in X(r-2)$, $\mathbf{x}^\top A_\eta \mathbf{x}=\lambda_1(A_\eta)$.
\end{claim}
\begin{proof}
	By \eqref{eq-12}, we have
	\[\begin{aligned}
		\sum_{\eta\in X(r-2)}\lambda_1(A_\eta)=\mathbf{x}^\top\left(\sum_{\eta\in X(r-2)}A_\eta\right)\mathbf{x}&=\sum_{\eta\in X(r-2)}\mathbf{x}^\top A_\eta\mathbf{x}\le \sum_{\eta\in X(r-2)}\lambda_1(A_\eta).
	\end{aligned}\]
Since every summand on the left-hand side is bounded above by the corresponding summand on the right-hand side, equality of the sums
forces equality term by term. Hence, $\mathbf{x}^\top A_\eta \mathbf{x}=\lambda_1(A_\eta)$ for every $\eta\in X(r-2)$.
\end{proof}
\begin{claim}\label{f-4}
	\rm	For every active $\eta$, $\operatorname{supp}\left(\mathbf{x}\right)\subseteq\mathcal{F}_\eta$.
\end{claim} 
\begin{proof}
	Let $\eta$ be an active simplex. Then $\mathcal{F}_\eta\ne \emptyset$. By \eqref{eq-local-diagonal}, we have $\operatorname{tr}\left(A_\eta\right)>0$. Moreover, Lemma~\ref{lem-1} implies that $A_\eta$ is orthogonally
	similar to $L_1^{\operatorname{down}}(H_\eta)\oplus O$, and hence $A_\eta$ is positive semidefinite. Therefore, we have $\lambda_1(A_\eta)>0$. Since $A_\eta$ is symmetric, the equality case of the Rayleigh quotient together with Claim~\ref{f-3} gives $A_\eta \mathbf{x}=\lambda_1(A_\eta)\mathbf{x}$ for every active $\eta$. By the definition of $A_\eta$, all rows and columns outside the
	coordinate set $\mathcal{F}_\eta$ are zero. Thus, for every $\sigma\in X(r)\setminus\mathcal{F}_\eta$, we have
	\[0=(A_\eta \mathbf{x})_\sigma=\lambda_1(A_\eta)\mathbf{x}_\sigma.\]
	Since $\lambda_1(A_\eta)>0$, we obtain $\mathbf{x}_\sigma=0$, and thus $\sigma\notin \operatorname{supp}\left(\mathbf{x}\right)$. Therefore, $\operatorname{supp}\left(\mathbf{x}\right)\subseteq\mathcal{F}_\eta$.
\end{proof}

Choose $\sigma_0\in X(r)$ such that $\mathbf{x}_{\sigma_0}\neq 0$. By Claim \ref{f-4}, we have $\sigma_0\in \operatorname{supp}\left(\mathbf{x}\right)\subseteq\mathcal{F}_\eta$, and thus $\eta\subset \sigma_0$ for every active $\eta$. Let $\sigma\in X(r)$ be arbitrary. Then every $(r-2)$-simplex of $\sigma$ is active and hence is contained in $\sigma_0$. Since $r\ge 2$, every vertex of $\sigma$ belongs to an $(r-2)$-simplex of $\sigma$. Therefore, $\sigma\subseteq\sigma_0$. Since both $\sigma$ and $\sigma_0$ have cardinality $r+1$, it follows that $\sigma=\sigma_0$. As $\sigma$ was arbitrary, $\sigma_0$ is the unique $r$-simplex of $X$. Consequently,  $\operatorname{rank}B_r(X)=1$ and $\Delta_{r-1}(X)=1$, which proves the desired conclusion for $q=1$.

\vspace{10pt}
\noindent\textbf{Case 2.} $2\le  q<m$. 
\vspace{10pt}

The key step in this case is to show that every local matrix $A_\eta$ has rank at most $q$. We prove this in the following claim.
\begin{claim}\label{c-2}
For every $\eta\in X(r-2)$, $\operatorname{rank}A_\eta\le  q$.
\end{claim}
\begin{proof}
Suppose, to the contrary, that \[\mathcal S=\{\eta\in X(r-2):\operatorname{rank}A_\eta>q\}\ne \emptyset.\]
	For every $\eta\in\mathcal S$, Lemma~\ref{lem-common-subspace}(ii)
	and \eqref{eq-15} give
	\[
	 U\subseteq\operatorname{im}A_\eta
	 \subseteq\mathbb R^{\mathcal F_\eta}.
	\]
	Hence, by \eqref{eq-coordinate-intersection},
	\[
	 U\subseteq
	 \mathbb R^{\cap_{\eta\in\mathcal S}\mathcal F_\eta}.
	\]
	Since $\dim U=q\ge2$, it follows that
	\begin{equation}\label{eq-common-facets}
	 \left|\bigcap_{\eta\in\mathcal S}\mathcal F_\eta\right|
	 \ge q\ge2.
	\end{equation}
	Let
	\[
	 C=\bigcup_{\eta\in\mathcal S}\eta.
	\]
	Every $r$-simplex in
	$\cap_{\eta\in\mathcal S}\mathcal F_\eta$ contains $C$.
	By \eqref{eq-common-facets}, choose distinct
	$\sigma_1,\sigma_2\in\cap_{\eta\in\mathcal S}\mathcal F_\eta$.
	Then
	\[
	 C\subseteq\sigma_1\cap\sigma_2.
	\]
	Two distinct $r$-simplices have intersection of cardinality at most
	$r$, and therefore $|C|\le r$. On the other hand, every element of
	$\mathcal S$ has cardinality $r-1$, so $|C|\ge r-1$. Consequently,
	\begin{equation}\label{eq-C-two-sizes}
	 |C|\in\{r-1,r\}.
	\end{equation}
	
By the definition of $C$, we have $\mathcal S\subseteq\binom{C}{r-1}$, and thus $|\mathcal S|\le \binom{|C|}{r-1}\le  r$. Since $\dim X=r$, there exists at least one $r$-simplex in $X$. Therefore, we have $f_r(X)\ge 1$. This implies \[f_{r-2}(X)\ge \binom{r+1}{r-1}>r\ge |\mathcal S|.\] It leads to $X(r-2)\setminus\mathcal S\ne\emptyset$.
Now take $\zeta\in X(r-2)\setminus\mathcal S$. Then $\operatorname{rank}A_\zeta\le  q$.
By Lemma~\ref{lem-common-subspace} again, we have $\operatorname{im}A_\zeta\subseteq U$. Define $\mathcal F_C=
\{\sigma\in X(r):C\subseteq\sigma\}$. Since every $r$-simplex in $\bigcap_{\eta\in\mathcal S}\mathcal F_\eta$ contains $C$, we have $\bigcap_{\eta\in\mathcal S}\mathcal F_\eta\subseteq\mathcal F_C$. It follows that \begin{equation}\label{eq-14}
	\operatorname{im}A_\zeta\subseteq U\subseteq\mathbb R^{\cap_{\eta\in\mathcal S}\mathcal F_\eta}\subseteq\mathbb R^{\mathcal F_C}.
\end{equation}
Let $\sigma\in\mathcal F_\zeta$. From \eqref{eq-14}, we have $A_\zeta e_\sigma\in\mathbb R^{\mathcal F_C}$. By noticing that $\left(A_\zeta e_\sigma\right)_{\sigma}=(A_\zeta)_{\sigma,\sigma}=2\ne 0$ from \eqref{eq-local-diagonal}, we have $\sigma\in \mathcal F_C$. 
Therefore, we obtain the following fact:
\begin{fact}\label{f-5}
	\rm	For every $\zeta\notin\mathcal S$, $\mathcal F_\zeta\subseteq\mathcal F_C$.
\end{fact} 
Let $\sigma\in X(r)$. Then $\sigma$ contains $\binom{r+1}{2}$ $(r-2)$-simplices. Since $|\mathcal S|\le r$ and $\binom{r+1}{2}>r$ for $r\ge 2$, there exists $\zeta\subset\sigma$ with $\zeta\notin\mathcal S$. By Fact \ref{f-5}, we have $\sigma\in\mathcal F_\zeta\subseteq\mathcal F_C$.
By the definition of $\mathcal F_C$, the following fact holds.
\begin{fact}\label{f-6}
	\rm	For every $\sigma\in X(r)$, $C\subseteq\sigma$.
\end{fact} 

By \eqref{eq-C-two-sizes}, we distinguish two possibilities.

\vspace{10pt}

\noindent {(i) $|C|=r-1$.}

\vspace{10pt}

In this case, Fact~\ref{f-6} implies that $C$ is an $(r-2)$-simplex, and every $r$-simplex is therefore of the form $C\cup\{u,v\}$. Thus, every $r$-simplex of $X$ corresponds to the edges of the link graph $H_C$. Since $C=\bigcup_{\eta\in\mathcal S}\eta$ and $|\eta|=r-1=|C|$ for every $\eta\in \mathcal S$, we have $\mathcal S=\{C\}$. Let $w$ be a non-isolated vertex of $H_C$ and choose $c\in C$.
Let 
\[\zeta=(C\setminus\{c\})\cup\{w\}.\]
Then $\zeta \in X(r-2)\setminus\mathcal S$. By Fact~\ref{f-6}, every edge of $H_\zeta$ contains the vertex $c$.
Hence, the non-isolated part of $H_\zeta$ is a star centered at $c$. By Lemma~\ref{lem-1}, $A_\zeta$ is orthogonally similar to
the down-Laplacian of a star graph. Since the incidence matrix of a forest has full column rank, we obtain $\operatorname{rank}A_\zeta=|\mathcal F_\zeta|$. Since $\operatorname{im}A_\zeta\subseteq\mathbb R^{\mathcal F_\zeta}$ from \eqref{eq-15}, we have $\operatorname{im}A_\zeta=\mathbb R^{\mathcal F_\zeta}$. Combining this with \eqref{eq-14}, we obtain $\mathbb R^{\mathcal F_\zeta}=\operatorname{im}A_\zeta\subseteq U$.
As $w$ ranges over all non-isolated vertices of $H_C$, the sets $\mathcal F_\zeta$ cover $X(r)$. Therefore, we have $\mathbb R^{X(r)}\subseteq U$. This implies  $m=\dim\mathbb R^{X(r)}\le  \dim U=q$, which contradicts $q<m$.

\vspace{15pt}

\noindent{(ii) $|C|=r$.}

In this case, Fact~\ref{f-6} implies that $C$ is an $(r-1)$-simplex, and every $r$-simplex is therefore of the form
$C\cup\{v\}$ for some $v\notin C$. Let $W=\{v:C\cup\{v\}\in X(r)\}$. Then $|W|=m$. For each $c\in C$, write $\eta_c=C\setminus\{c\}$. The link graph $H_{\eta_c}$ is a star graph with center $c$ and leaf set $W$, and hence $H_{\eta_c}\cong K_{1,m}$. Since the incidence matrix of a tree has full column rank, Lemma~\ref{lem-1} implies that $\operatorname{rank}A_{\eta_c}=m>q$. Therefore, $\eta_c\in\mathcal S$. Fix $v\in W$. Choose distinct $c,d\in C$ and let
\[\zeta=(C\setminus\{c,d\})\cup\{v\}.
\]
Then the only $r$-simplex containing $\zeta$ is $C\cup\{v\}$, and hence $\mathcal F_\zeta=\{C\cup\{v\}\}$.  By \eqref{eq-15}, we have $\operatorname{im}A_\zeta\subseteq\mathbb R^{\mathcal F_\zeta}=\operatorname{span}\{e_{C\cup\{v\}}\}$. Moreover, $\mathcal F_\zeta\neq\emptyset$, and therefore $A_\zeta\neq0$. Hence $\operatorname{rank}A_\zeta\ge 1$, which yields
$\operatorname{im}A_\zeta=\operatorname{span}\{e_{C\cup\{v\}}\}$. Since $v\notin C$, we have $\zeta\nsubseteq C$. Every element of
$\mathcal S$ is contained in $C$, and therefore $\zeta\notin\mathcal S$. Combining this with \eqref{eq-14}, we have $e_{C\cup\{v\}}\in \operatorname{im}A_\zeta\subseteq U$. As $v$ is arbitrary, the simplices $C\cup\{v\}$ exhaust $X(r)$.
Consequently, $\mathbb R^{X(r)}\subseteq U$, which implies $m=\dim\mathbb R^{X(r)}\le  \dim U=q$,
contradicting the assumption $q<m$.

Both possibilities lead to contradictions. Hence $\mathcal S=\emptyset$, which proves the claim.
\end{proof}
For every $\eta\in X(r-2)$, it follows from Claim \ref{c-2} that $\operatorname{rank}A_\eta\le  q$. Hence, by part (i) of Lemma~\ref{lem-common-subspace}, we have $\operatorname{im}A_\eta\subseteq U$. Therefore, $\operatorname{im}\left(\sum_{\eta\in X(r-2)}A_\eta\right)\subseteq U$, and thus
\[\operatorname{rank}\left(\sum_{\eta\in X(r-2)}A_\eta\right)=\dim\operatorname{im}\left(\sum_{\eta\in X(r-2)}A_\eta\right)\le \dim U= q.\]
By Lemma~\ref{lem-2}, we have
	\[\operatorname{rank}L_r^{\operatorname{down}}(X)=\operatorname{rank}\left(\sum_{\eta\in X(r-2)}A_\eta\right)\le q.\]
Since $L_r^{\operatorname{down}}(X)=B_r(X)^{\top}B_r(X)$, we obtain \[\operatorname{rank}B_r(X)=\operatorname{rank}L_r^{\operatorname{down}}(X)\le q.\]
It remains to prove $\Delta_{r-1}(X)\le q$. For every $\eta\in X(r-2)$, since $\operatorname{rank}A_\eta\le  q$ and $A_\eta$ is positive
semidefinite, we have $\operatorname{tr}A_\eta=\Phi_q(A_\eta)$. Combining this with Proposition~\ref{prop-equality-chain}, we obtain \[
\operatorname{tr}A_\eta=\sum_{\tau\in X(r-1)\colon \eta\subset\tau}\min\left\{\deg_{r-1}(X;\tau),q\right\}.
\]
On the other hand, by \eqref{eq-local-diagonal} and double counting, we have
$\operatorname{tr}A_\eta=2|\mathcal F_\eta|=\sum_{\eta\subset\tau}\deg_{r-1}(X;\tau)$. Therefore,
\[\sum_{\eta\subset\tau}\left(\deg_{r-1}(X;\tau)-\min\{\deg_{r-1}(X;\tau),q\}\right)=0.\]
Since every summand is nonnegative, we conclude that \[\deg_{r-1}(X;\tau)=\min\{\deg_{r-1}(X;\tau),q\}\]  for every $\tau$ satisfying $\eta\subset\tau$. Hence, $\deg_{r-1}(X;\tau)\le q$ for every such $\tau$. Finally, every $(r-1)$-simplex of $X$ contains an $(r-2)$-simplex.
Since $\eta\in X(r-2)$ was arbitrary, the above inequality applies to every $\tau\in X(r-1)$. Therefore,  $\Delta_{r-1}(X)\le q$.

This proves the necessity, and hence the theorem follows.
\end{proof}
	
We now derive the corollaries stated in Section~\ref{sc-1}. We begin with the equality case for the largest eigenvalue.
\begin{proof}[Proof of Corollary~\ref{cor-first}]
By Theorem~\ref{thm-main} with $q=1$, we have $\lambda_{r-1,1}(X)=d_{r-1,1}^{\top}(X)$ if and only if $\operatorname{rank}B_r(X)\le 1$ and
$\Delta_{r-1}(X)\le 1$. Therefore, it remains to prove that $\operatorname{rank}B_r(X)\le 1$ and
$\Delta_{r-1}(X)\le 1$ hold if and only if $f_r(X)=1$.
	
Suppose that $\operatorname{rank}B_r(X)\le 1$ and $\Delta_{r-1}(X)\le 1$. Since $X$ is $r$-dimensional, we have $f_r(X)\ge 1$.
	If $f_r(X)\ge 2$, then either two distinct $r$-simplices share an $(r-1)$-face, in which case $	\Delta_{r-1}(X)\ge 2$, or they have no common $(r-1)$-face. In the latter case, the corresponding columns of $B_r(X)$ have disjoint supports and are therefore linearly independent, which implies $\operatorname{rank}B_r(X)\ge 2$. Both cases contradict the assumptions. Hence, $f_r(X)=1$.
	
Conversely, assume that $f_r(X)=1$. Then $B_r(X)$ consists of a single nonzero column, and hence $\operatorname{rank}B_r(X)=1$. Moreover, every $(r-1)$-simplex is contained in at most one $r$-simplex. Thus, $\Delta_{r-1}(X)=1$. 
	
	This completes the proof.
\end{proof}
	
We next extend the above characterization from the largest eigenvalue to the entire sequences.
\begin{proof}[Proof of Corollary~\ref{cor-full}]
	Suppose first that $\lambda_{r-1}(X)=d_{r-1}^{\top}(X)$. Then, by considering the first entries of the two sequences, we obtain $\lambda_{r-1,1}(X)=d_{r-1,1}^{\top}(X)$. By Corollary~\ref{cor-first}, we obtain $f_r(X)=1$.

	Conversely, suppose that $f_r(X)=1$. Let $X(r)=\{\sigma\}$. By Theorem \ref{thm-13}, we have \[L_{r}^{\operatorname{down}}(X)=\begin{pmatrix}r+1\end{pmatrix}.\] It follows from \eqref{eq-01} that $	\lambda_{r-1}(X)=(r+1,0,\ldots,0)$.
	On the other hand, the only $(r-1)$-simplices of positive degree are $r+1$ faces of $\sigma$, each having degree one. Hence, we have
$d_{r-1}^{\top}(X)=(r+1,0,\ldots,0)$. Therefore, $\lambda_{r-1}(X)=d_{r-1}^{\top}(X)$.

For a non-pure complex, additional maximal faces of dimension smaller than
$r$ may still be present; they contribute only zero entries to the two
sequences considered here. In particular, if $X$ is pure, then
$f_r(X)=1$ means that $X$ has a unique facet, which is an $r$-simplex.
Hence, $X\cong\Delta_r$.
The converse is immediate.
\end{proof}

Finally, we derive the homological reformulation of Theorem~\ref{thm-main}.
\begin{proof}[Proof of Corollary~\ref{cor-homology}]
Since $X$ has dimension $r$, we have $C_{r+1}(X)=0$, and hence \[\widetilde{H}_r(X;\mathbb{R})=\ker\partial_r(X).\]
By the rank-nullity theorem,
\[\operatorname{rank}B_r(X)=\operatorname{rank}\partial_r(X)=f_r(X)-\dim\ker\partial_r(X)=f_r(X)-\beta_r(X).
\]
The result now follows immediately from Theorem~\ref{thm-main}.
\end{proof}

The following remark highlights a fundamental difference between the higher-dimensional case and the graph case.
	\begin{remark}\label{rem-no-intermediate}
		For $r\ge 2$, Theorem~\ref{thm-main} implies that
		\[\sum_{i=1}^{q}\lambda_{r-1,i}(X)=\sum_{i=1}^{q}d_{r-1,i}^{\top}(X)\]
		if and only if $\lambda_{r-1,q+1}(X)=0$ and $d_{r-1,q+1}^{\top}(X)=0$.
		Thus, there is no genuine intermediate partial-sum equality:
		whenever equality occurs, both sequences have already exhausted
		all their nonzero terms by the $q$-th position.
		
		This is sharply different from the graph case, where nontrivial
		intermediate equality cases occur and were completely
		characterized by Cai, Chen, Yang, and Zhang~\cite{CCYZ2026}.
	\end{remark}

	\section*{Declaration of competing interest}
	
	The authors declare that they have no competing interests.
	
	\section*{Acknowledgements}
	
	This work is supported by the National Natural Science Foundation of
	China (No.~12371362).
	
%	\clearpage
%	\addcontentsline{toc}{section}{References}


{\small
	\begin{thebibliography}{99}
	\bibitem{B2011} H. Bai, The Grone--Merris conjecture, Trans. Amer. Math. Soc. 363 (2011), 4463--4474.
	
	\bibitem{CCYZ2026} D. Cai, Z. Chen, J. Yang, X. D. Zhang, The equality cases for the Grone--Merris--Bai theorem, arXiv:2607.23583, 2026.
	
	%\bibitem{DW2002} X. Dong, M. L. Wachs, Combinatorial Laplacian of the matching complex, Electron. J. Combin. 9 (2002), \#R17.
	
	\bibitem{DR2002} A. M. Duval, V. Reiner, Shifted simplicial complexes are Laplacian integral, Trans. Amer. Math. Soc. 354 (2002), 4313--4344.
	
	%\bibitem{F1996} J. Friedman, Computing Betti numbers via combinatorial Laplacians, in: Proceedings of the Twenty-Eighth Annual ACM Symposium on Theory of Computing, 1996, pp. 386--391.
	\bibitem{F2015} S. Friedland, Equality in Wielandt's eigenvalue inequality, Special Matrices 3 (2015), 53--57.
	\bibitem{FWW2024} Y. Z. Fan, H. F. Wu, Y. Wang, The largest Laplacian eigenvalue and the balancedness of simplicial complexes, J. Algebraic Combin. 61 (2025), Art. 53.
	\bibitem{GM1994} R. Grone, R. Merris, The Laplacian spectrum of a graph II, SIAM J. Discrete Math. 7 (1994), 221--229.

	\bibitem{G2026} V. Gupta, Spectral bounds and shifted complexes: eigenvalues of the up-Laplacian via face degrees, arXiv:2608.01694, 2026.
	
	\bibitem{HJ2012} R. A. Horn, C. R. Johnson, Matrix Analysis, 2nd ed., Cambridge University Press, Cambridge, 2012.
	
	%\bibitem{HLW2026} Y. Han, L. Lu, J. Wang, On the sum of the two largest eigenvalues of the curl-curl operator on graphs, arXiv:2606.26512, 2026.
	
	\bibitem{HJ2013} D. Horak, J. Jost, Spectra of combinatorial Laplace operators on simplicial complexes, Adv. Math. 244 (2013), 303--336.
	
	\bibitem{H2026} J. Huang, The Duval--Reiner conjecture: counterexamples and the second partial-sum inequality, arXiv:2607.20051, 2026.
	
	\bibitem{KRS2000} W. Kook, V. Reiner, D. Stanton, Combinatorial Laplacians of matroid complexes, J. Amer. Math. Soc. 13 (2000), 129--148.
	
	\bibitem{L2020} A. Lew, Spectral gaps, missing faces and minimal degrees, J. Combin. Theory Ser. A 169 (2020), 105127.
	
	\bibitem{M1994} R. Merris, Degree maximal graphs are Laplacian integral, Linear Algebra Appl. 199 (1994), 381--389.
	\bibitem{MOA2011} A. W. Marshall, I. Olkin, B. C. Arnold, Inequalities: Theory of Majorization and Its Applications, Springer Series in Statistics, Springer, New York, 2nd ed., 2011.
	\bibitem{SY2020} S. Shukla, D. Yogeshwaran, Spectral gap bounds for the simplicial Laplacian and an application to random complexes,
	J. Combin. Theory Ser. A 169 (2020), 105134.
	
	\bibitem{ZHL2026} X. F. Zhan, X. Y. Huang, H. Q. Lin, Proof of Lew's conjecture on the spectral gaps of simplicial complexes,
	J. Combin. Theory Ser. A 217 (2026), 106091.
	
	\bibitem{ZSF2026} H. Z. Zhang, Y. M. Song, Y. Z. Fan, Degree majorization and Laplacian eigenvalue sums for simplicial complexes,
	arXiv:2607.20910, 2026.
	
\end{thebibliography}
	}
\end{document}